\documentclass[12pt]{amsart}
\usepackage{amsmath,amssymb,amsthm}
\usepackage{mathtools}
\usepackage{scalerel}
\usepackage{enumerate}
\usepackage{hyperref}
\hypersetup{colorlinks=true,citecolor=blue,linkcolor=red}
\usepackage{cite}
\usepackage{color}
\usepackage[usestackEOL]{stackengine}
\usepackage[a4paper,width=16.5cm,top=3cm,bottom=2.8cm]{geometry}

\newtheorem{theorem}{Theorem}[section]
\newtheorem{proposition}[theorem]{Proposition}
\newtheorem{corollary}[theorem]{Corollary}
\newtheorem{lemma}[theorem]{Lemma}

\newtheorem{remark}[theorem]{Remark}

\newtheorem{question}[theorem]{Question}

\theoremstyle{plain}
\newtheorem*{theorem*}{Theorem}
\newtheorem*{question*}{Question}
\newtheorem*{example*}{Example}

\def\dashint{\,\ThisStyle{\ensurestackMath{%
			\stackinset{c}{.2\LMpt}{c}{.5\LMpt}{\SavedStyle-}{\SavedStyle\phantom{\int}}}%
		\setbox0=\hbox{$\SavedStyle\int\,$}\kern-\wd0}\int}
\newcommand{\rn}{\mathbb{R}^n}

\newcommand{\honern}{{H}^1(\mathbb{R}^n)}

\newcommand{\hprn}{{H}^p(\mathbb{R}^n)}

\newcommand{\bmorn}{{BMO}(\rn)}
\newcommand{\bmor}{{BMO}(\mathbb{R})}
\newcommand{\bmortwo}{{BMO}(\mathbb{R}^2)}

\newcommand{\bmornn}[1]{\|#1\|_{{BMO}(\rn)}}
\newcommand{\bmoronen}[1]{\|#1\|_{{\bmor}}}

\newcommand{\bmortwon}[1]{\|#1\|_{{\bmortwo}}}

\newcommand{\bmodyadicrn}{{BMO_d}(\rn)}

\newcommand{\vmorn}{{VMO}(\rn)}

\newcommand{\lonen}[1]{\|#1\|_{{L}^1}}

\newcommand{\lprn}{{L}^p(\mathbb{R}^n)}

\newcommand{\linfty}{{L}^{\infty}}
\newcommand{\linftyrn}{{L}^{\infty}(\mathbb{R}^n)}

\newcommand{\linftyn}[1]{\|#1\|_{\linfty}}

\newcommand{\lonernloc}{{L}_{loc}^1(\mathbb{R}^n)}

\newcommand{\ave}[2]{\dashint_{#2}{#1}}

\newcommand{\osc}[2]{O(#1,#2)}

\newcommand{\woneprn}{W^{1,p}(\rn)}

\newcommand{\homogenlipschit}[2]{\dot{\Lambda}^{#1}(\mathbb{R}^{#2})}

\begin{document}
	
	\title{Maximal Operators on BMO and Slices}
	\author{Shahaboddin Shaabani}
	\address{Department of Mathematics and Statistics\\Concordia University}
	\email{shahaboddin.shaabani@concordia.ca}
	
	\begin{abstract}
		We prove that the uncentered Hardy-Littlewood maximal operator is discontinuous on $\bmorn$ and maps $\vmorn$ to itself. A counterexample to the boundedness of the strong and directional maximal operators on $\bmorn$ is given, and properties of slices of $\bmorn$ functions are discussed.
	\end{abstract}	
	\date{}
	\maketitle
	\section{Introduction}	
	Let $A\subset \rn$ be a measurable set with positive finite measure and $f\in \lonernloc$. By the mean oscillation of $f$ on $A$ we mean the quantity
	
	\[
	\osc{f}{A}:=\ave{|f-f_A|}{A},
	\]
	where $\ave{f}{A}$ and $f_A$ mean the average of $f$ over $A$, i.e. $\frac{1}{|A|}\int_{A} f$.
	Then it is said that $f$ is of bounded mean oscillation if $\osc{f}{Q}$ is uniformly bounded on all cubes $Q$ (by a cube we mean a closed cube with sides parallel to the axes). The space of such functions is denoted by $\bmorn$, and modulo constants the following quantity defines a norm on this space:
	
	\[
	\bmornn{f}:=\sup_{Q}O(f,Q).
	\]
	Sometimes we use $\|f\|_{BMO(Q_0)}$ which means that we take the above supremum over all cubes contained in $Q_0$. $\bmorn$ is a Banach space, and since its introduction has played an important role in harmonic analysis. It is the dual of the Hardy space $\honern$, it contains $\linftyrn$ and somehow serves as a substitute for it. For instance, Calderón-Zygmund singular integral operators map $\bmorn$ to itself and consequently these operators map $\linftyrn$ to $\bmorn$ but not to itself \cite{MR0807149}.\\
	
	Another important class of operators is the class of maximal operators, and the first objective of the present paper is to investigate the action of some of these operators on $\bmorn$. Let us recall that the uncentered Hardy-Littlewood maximal operator is defined by
	\[
	Mf(x):=\sup_{x\in Q}\ave{|f|}{Q}, \quad f\in\lonernloc,
	\]
	where the above supremum is taken over all cubes containing $x$. As it is well known $M$ is of weak-type $(1,1)$ and bounded on $\lprn$ for $1<p\le \infty$ \cite{MR0807149}. For a function $f \in \bmorn$, it might be the case that $Mf$ is identically equal to infinity. For instance, this is the case when $f(x)=\log|x|$. However, in \cite{MR0621018} the authors proved that if this is not the case then $Mf$ belongs to $\bmorn$ and for a dimensional constant $c(n)$ we have
	\[
	\bmornn{Mf}\le c(n)\bmornn{f}.
	\]
	Another proof of this was given in \cite{MR0660603}, and a third one in \cite{MR4012803}, where the author proved that $M$ preserves Poincaré inequalities. Regarding this, we ask the following question about the continuity of the uncentered Hardy-Littlewood maximal operator on $\bmorn$:
	\begin{question}
		Let $f\in \linftyrn$ and $\{f_k\}$ be a sequence of bounded functions converging to $f$ in $\bmorn$. Is it true that $\{Mf_k\}$ converges to $Mf$ in $\bmorn$?
	\end{question}
	
	The operator $M$ is nonlinear and for such operators continuity does not follow from boundedness. However, it is pointwise sublinear and this makes it continuous on $\lprn$ for $1<p\le\infty$. In \cite{MR2280193}, a similar question has been studied for Sobolev spaces, where the author proved that $M$ is continuous on $\woneprn$ for $1<p<\infty$. However, in Section 2 we give a negative answer to the above question.\\
	
	$\bmorn$ has an important subspace, namely $\vmorn$ or functions of vanishing mean oscillation. $\vmorn$ is the closure of the uniformly continuous functions in $\bmorn$. Another characterization of $\vmorn$ is given in terms of the modulus of mean oscillation which is defined by
	
	\begin{equation}\label{vmocondition}
		\omega(f, \delta):=\sup_{l(Q)\le \delta}O(f,Q),
	\end{equation}
	and $f\in \vmorn$ exactly when $\lim\limits_{\delta\to0}\omega(f, \delta)=0$ (in the above by $l(Q)$ we mean the side length of $Q$)\cite{MR0377518}. Regarding this subspace we ask:
	\begin{question}
		Let $f\in \vmorn$ such that $Mf$ is not identically equal to infinity. Is it true that $Mf \in \vmorn$?
	\end{question} 
	
	In Section 3 we provide a positive answer to this question.\\
	
	In the last section, we consider the action of some other maximal operators on $\bmorn$. More specifically, the directional maximal operator in the direction $e_1=(1,0,...,0)$, $M_{e_1}$, and the strong maximal operator, $M_s$, which are defined as the following:
	
	\begin{equation*}
		M_{e_1}f(x_1,x'):=\sup_{x_1 \in I} \ave{|f_{x'}|}{I}, \quad \quad M_sf(x):=\sup_{x \in R} \ave{|f|}{R}.
	\end{equation*}
	In the above, $f_{x'}(t):=f(t,x')$, where $(t,x')\in\mathbb{R}\times\mathbb{R}^{n-1}$, and the left supremum is taken over all closed intervals containing $x_1$. In a similar way one can define the directional maximal operator $M_e$, which is taken in the direction $e\in\mathbb{S}^{n-1}$, simply by taking the one dimensional uncentered Hardy-Littlewood maximal operator on every line in direction $e$. However, since $\bmorn$ is invariant under rotations it is enough to study $M_{e_1}$. In the above the right supremum is taken over all rectangles containing $x$, and by a rectangle we mean a closed rectangle with sides parallel to the axes. These are the most important maximal operators in multi-parameter harmonic analysis and are bounded and continuous on $\lprn$ for $1<p\le\infty$ \cite{MR0766959}. Regarding these operators we ask:
	
	\begin{question}
		Are there constants $C, C'\ge1$ such that at least for every bounded function $f$ the following inequalities hold?
		\[
		\bmornn{M_{e_1}f}\le C\bmornn{f}, \quad \bmornn{M_sf}\le C'\bmornn{f}.
		\]
	\end{question}
	
	To answer this question we have to study the properties of slices of functions in $\bmorn$, which is the second objective of this paper. Many function spaces have the property that their slices lie in the same scale of spaces. For example, almost every slice of a function in $\lprn$ or $\woneprn$ lies in $L^p(\mathbb{R}^{n-1})$ or $W^{1,p}(\mathbb{R}^{n-1})$, respectively\cite{MR4567945}. The same is true for $BMO_s(\rn)$, strong $BMO$, which is the subspace of $\bmorn$ consisting of all functions with bounded mean oscillation on rectangles\cite{MR4085448}. This property is also satisfied by the scale of homogeneous Lipschitz spaces $\homogenlipschit{n(\frac{1}{p}-1)}{n}$, the duals of $\hprn$ for $0<p<1$ \cite{MR0807149}. Regarding this, we ask our last question:
	\begin{question}
		Is it true that almost every horizontal or vertical slice of a function in $\bmortwo$ belongs to $\bmor$?
	\end{question}
	In Section 4 we answer both questions negatively and in the last theorem of this paper we prove a property of the slices of functions in $\bmortwo$.\\
	
	Before we proceed further, let us fix some notation. By $A\lesssim B$, $A\gtrsim B$ and $A\approx B$, we mean $A\le CB$, $A\ge CB$ and $C^{-1}B\le A \le CB$ respectively, where $C$ is a constant independent of the important parameters.

\section{Discontinuity of $M$ on $\bmorn$}
	Our theorem in this section is the following:
	
	\begin{theorem}\label{discontinuitytheorem}
		Let $f$ be a nonnegative function supported in $[0,1]$,  $\linftyn{f}\le1$ and $\lonen{f}>\log2$. Then, there exists a sequence of bounded functions $\{f_n\}$ converging to $f$ in $\bmor$ such that $\{Mf_n\}$ does not converge to $Mf$ in $\bmor$.
	\end{theorem}	
	To prove this we need a couple of simple lemmas which we give below.
	\begin{lemma}\label{evenperiodicextension}
		Let $T>0$ and $h \in BMO[0,\frac{T}{2}]$. Then the even periodic extension of $h$, which is defined by
		
		\begin{equation*}
			H(x):=h(x),\quad x\in[0,\frac{T}{2}], \quad\quad	H(-x)=H(x),\quad H(x+T)=H(x), \quad x \in \mathbb{R},
		\end{equation*}
		is in $\bmor$ and $\bmoronen{H}\le 10\|h\|_{BMO[0,\frac{T}{2}]}$.
	\end{lemma}	
	\begin{proof}
		For an arbitrary interval $I$, there are two possibilities:\\
		
		(i) $|I|\le \frac{T}{2}$.\\
		
		In this case by a translation by an integer multiple of $T$ and using periodicity of $H$, we may assume either $I\subset[-\frac{T}{2},\frac{T}{2}]$ or $I\subset[0,T]$. Suppose $I\subset[-\frac{T}{2},\frac{T}{2}]$ and note that if $0\notin I$, we have either $I\subset[0,\frac{T}{2}]$ or $I\subset[-\frac{T}{2},0]$ and from the symmetry $O(H,I)\le\|h\|_{BMO[0,\frac{T}{2}]}$. If $0\in I$, then take the interval $J$ centered at zero with the right half $J^+$, which contains $I$ and $|J|\le2 |I|.$ Again from symmetry we get
		\[
		O(H,I)\le2 \frac{|J|}{|I|}O(H,J)\le4O(H,J^+) \le4\|h\|_{BMO[0,\frac{T}{2}]}.
		\]
		The same argument works for $I\subset [0,T]$. This time we use the symmetry of $H$ around $\frac{T}{2}$.\\
		
		(ii) $|I|\ge \frac{T}{2}$.\\
		
		This time take $J=[nT,mT]$ with $n,m \in \mathbb{Z}$ which contains $I$ and $|J|\le |I|+2T\le 5|I|$. And again like the previous cases, from the symmetry and periodicity of $H$, we get
		\[
		O(H,I)\le2 \frac{|J|}{|I|}O(H,J)\le10O(h,[0,\frac{T}{2}]) \le10\|h\|_{BMO[0,\frac{T}{2}]}.
		\]
		
		The proof is now complete.
	\end{proof}
	In the above, the norm of the extension operator is independent of $T$, and we will use this in the proof of the next lemma.
	
	\begin{remark}
		There are much more general ways to extend $BMO$ functions to the outside of domains, but for the purpose of our paper the above simple lemma is enough. See \cite{MR0554817} for more on extensions. 
	\end{remark}
	
	\begin{lemma}\label{g_nsequence}
		For $c<-1$, there exists a sequence of functions $\{g_n\}$, $n\ge1$ with the following properties:
		\begin{enumerate}
			\item $g_n\ge 0$
			\item $g_n=0$ on $[c,1]$
			\item $\linftyn{g_n}\le1$
			\item $\lim\limits_{n\to \infty}\ave{g_n}{[0,n]}=1$
			\item $\lim\limits_{n\to \infty}\bmoronen{g_n}=0$.
		\end{enumerate}
	\end{lemma}
	\begin{proof}
		Let $\log^+|x|=\max\{0, \log |x|\}$ be the positive part of the logarithm, and consider the function $h_n(x)=\log^+x$ on the interval $[0,n]$, which belongs to $BMO[0,n]$ with $\|h_n\|_{BMO[0,n]}\le \bmoronen{\log^+|\cdot|}$. Then an application of Lemma \ref{evenperiodicextension} with $T=2n$ gives us a sequence of nonnegative functions $H_n$ with $\bmoronen{H_n}\lesssim1$ (here our bounds are independent of $n$). Now, let $g_n=\frac{1}{1+\log n}H_n(x)\left(1-\chi_{[c,0]}(x)\right)$. Then, the first three properties are immediate from the definition, the forth one follows from integration, and the last one from
		\[
		\bmoronen{g_n}\le \frac{1}{1+\log n}\left(\bmoronen{H_n}+\linftyn{H_n\chi_{[c,0]}}\right)\lesssim\frac{\log|c|}{1+\log n} \quad n\ge1.
		\]
		This finishes the proof.
	\end{proof}

	Now, we turn to the proof of the above theorem.
	\begin{proof}[Proof of Theorem \ref{discontinuitytheorem}]
		Let $f$ be as in the theorem, $a=\int_{0}^{1}f$, $c<0$ a constant with large magnitude to be determined later, and let $g_n$ be the sequence constructed in Lemma \ref{g_nsequence}.\\
		
		We will show  that 
		\begin{equation}\label{liminfpositive}
			\varliminf\limits_{n\to \infty}\bmoronen{Mf_n-Mf}>0, \quad f_n:=f+\frac{a}{1-c}g_n, \quad n\ge1.
		\end{equation}
		This proves the theorem once we note that since $f$ and $g_n$ are bounded functions, $f_n$ is bounded too. Also, from the fifth property of $\{g_n\}$ in the above lemma, $\{f_n\}$ converges to $f$ in $\bmor$.\\
		
		To begin with, we claim that $Mf_n=Mf$ on $[c,0]$. To see this, note that from the positivity of $f$ and $g_n$,  $Mf_n(x)\ge Mf(x)$ for all values of $x$, and it remains to show that the reverse inequality holds also. For $x \in [c,0]$, $Mf(x)\ge \frac{\int_{c}^{1}f}{1-c}=\frac{a}{1-c}$, and for any interval $I$ which contains $x$, we have two possibilities:
		\begin{enumerate}[(i)]
			\item either $I\subset (-\infty,0)$, in which case from the third property of $g_n$ we have
			\[
			\ave{f_n}{I}=\ave{\left(f+\frac{a}{1-c}g_n\right)}{I}=\frac{a}{1-c}\ave{g_n}{I}\le \frac{a}{1-c}\linftyn{g_n}\le\frac{a}{1-c}\le Mf(x),
			\]
			\item or $I\cap[0,1]\ne \emptyset$, in which case the second and the third property of $g_n$ gives us
			\begin{align*}
				&\ave{f_n}{I}=\ave{\left(f+\frac{a}{1-c}g_n\right)}{I}=\frac{|I\cap[x,1]|}{|I|}\ave{f}{I\cap[x,1]}+\frac{|I\backslash[x,1]|}{|I|}\frac{a}{1-c}\ave{g_n}{I\backslash[x,1]}\\&\le \frac{|I\cap[x,1]|}{|I|} Mf(x)+ \frac{|I\backslash[x,1]|}{|I|}\frac{a}{1-c}\le Mf(x).
			\end{align*}
		\end{enumerate}
		This proves our claim. \\
		
		Next, we look at the mean oscillation of $Mf_n-Mf$ on $[2c,0]$. Because this function vanishes on $[c,0]$, we have
		
		\begin{equation}\label{meanlessthanoscillation}
			O(Mf_n-Mf,[2c,0]) \ge \frac{1}{4}\ave{(Mf_n-Mf)}{[2c,c]}.
		\end{equation}
		
		To bound the right hand side of the above inequality from below, we note that $0\le f\le \chi_{[0,1]}$ so $Mf(x)\le M(\chi_{[0,1]})(x)=\frac{1}{1-x}$ for $x\le0$. Also, for $x\le 0$ we have 
		\begin{equation*}
			Mf_n(x)=M\left(f+\frac{a}{1-c}g_n\right)(x)\ge\frac{a}{1-c}\ave{g_n}{[x,n]}\ge\frac{a}{1-c}.\frac{n}{n-x}\ave{g_n}{[0,n]}.
		\end{equation*}
		So we get the following estimate for the right hand side in \eqref{meanlessthanoscillation}:
		
		\begin{equation}\label{calculationofmean}
			\ave{\left(Mf_n-Mf\right)}{[2c,c]}\ge\frac{a}{1-c}\ave{g_n}{[0,n]}\ave{\frac{n}{n-x}dx}{[2c,c]}-\ave{\frac{1}{1-x}dx}{[2c,c]}.
		\end{equation}
		
		Combining \eqref{meanlessthanoscillation} and \eqref{calculationofmean}, gives us
		\begin{equation*}
			\bmoronen{Mf_n-Mf}\ge\frac{1}{4}\left(\frac{a}{1-c}\ave{g_n}{[0,n]}\ave{\frac{n}{n-x}dx}{[2c,c]}+\frac{1}{c}\log\left(1+\frac{c}{c-1}\right)\right).
		\end{equation*}
		Now, taking the limit inferior as $n\to\infty$ and using the forth property of $g_n$ gives us
		\begin{equation*}
			\varliminf\limits_{n\to \infty}\bmoronen{Mf_n-Mf}\ge\frac{1}{4}\left(\frac{a}{1-c}+\frac{1}{c}\log\left(1+\frac{c}{c-1}\right)\right).
		\end{equation*}
		This shows that if we have
		\begin{equation}\label{alargerthanlog}
			a>\frac{c-1}{c}\log\left(1+\frac{c}{c-1}\right)
		\end{equation}
		then \eqref{liminfpositive} holds. Here we note that the function on the right hand side of \eqref{alargerthanlog}, attains its minimum, which is $\log2$, at infinity. Also, from the assumption $a>\log2$, so if we choose $|c|$ sufficiently large \eqref{alargerthanlog} holds, and this completes the proof.
	\end{proof}
	By lifting the above functions to higher dimensions with
	
	\begin{equation}\label{lifting}
		f(x_1,..,x_n)=f(x_1), \quad g_m(x_1,..,x_n)=g_m(x_1),
	\end{equation}
	we obtain a counterexample for continuity of the $n$-dimensional uncentered Hardy-Littlewood maximal operator on $\bmorn$, simply because the $\bmorn$ norms and the maximal operator become one dimensional.
	\begin{corollary}
		The uncentered Hardy-Littlewood maximal operator is bounded on $\linftyrn$ equipped with the $BMO$ norm, but it is not continuous.
	\end{corollary}
	
	\section{The Uncentered Hardy-Littlewood Maximal Operator on $\vmorn$}
	As it was mentioned before, $\vmorn$ is the $\bmorn$-closure of the uniformly continuous functions which belong to $\bmorn$. The operator $M$ reduces modulus of continuity, because it is pointwise sublinear, so it preserves uniformly continuous functions. But from our previous result, one cannot deduce boundedness of $M$ on $\vmorn$ by a limiting argument. Nevertheless, we have the following theorem:
	
	\begin{theorem}\label{vmotheorem}
		Let $f \in \vmorn$ and suppose $Mf$ is not identically equal to infinity. Then $Mf$ belongs to $\vmorn$.
	\end{theorem}
	Before we prove this, we bring the following lemma which is needed later.
	\begin{lemma}\label{meanoscillatinonAlemma}
		Let $A$ be a measurable subset of a cube $Q$ of positive measure and $f\in \bmorn$ with $\bmornn{f}=1$; then we have
		\begin{equation}\label{jensen}
			\ave{\left|f-f_Q\right|}{A}\lesssim 1+\log\frac{|Q|}{|A|}.
		\end{equation}
	\end{lemma}
	\begin{proof}
		From the John-Nirenberg inequality \cite{MR0131498}, there is a dimensional constant $c>0$ such that
		\[
		\ave{e^{c|f-f_Q|}}{A}\le\frac{|Q|}{|A|}\ave{e^{c|f-f_Q|}}{Q}\lesssim \frac{|Q|}{|A|}.
		\]
		
		Now, Jensen's inequality gives us \eqref{jensen}, as follows:
		\[
		\ave{\left|f-f_Q\right|}{A}=\frac{1}{c}\ave{\log e^{c|f-f_Q|}}{A}\le \frac{1}{c}\log\ave{ e^{c|f-f_Q|}}{A}\lesssim 1+\log\frac{|Q|}{|A|}.
		\]
	\end{proof}
	\begin{remark}
		In the above lemma, let $A$ be a rectangle and take $Q$ to be the smallest cube which contains it. Then 
		\[
		O(f,A)\lesssim1+\log e(A),
		\] 
		where $e(A)$ is the eccentricity of $A$, or the ratio of the largest side to the smallest one.
	\end{remark}
	We now turn to the proof of Theorem \ref{vmotheorem}.
	
	\begin{proof}[Proof of Theorm \ref{vmotheorem}]
		Let $f$ be as in the theorem then we have to show that $\varlimsup\limits_{\delta\to 0}\omega(Mf,\delta)=0$. Now, for every cube Q, we have $O(|f|,Q)\le 2O(f,Q)$, which means that $|f|\in \vmorn$ too. From this together with $M(|f|)=Mf$, it is enough to prove the theorem for nonnegative functions. Also, from the homogeneity of $M$ we may assume $\bmornn{f}=1$.\\
		
		Let $Q_0$ be a cube and $c$ a constant with $c>e$. We decompose $M$ into the local part, $M_1$, and the nonlocal part,
		$M_2$, as follows:
		
		\begin{equation*}
			M_1f(x):=\underset{l(Q)\le cl(Q_0)}{\underset{x\in Q}\sup f_Q}, \quad \quad M_2f(x):=\underset{l(Q)\ge cl(Q_0)}{\underset{x\in Q}\sup f_Q}.
		\end{equation*}
		
		We have $Mf(x)=\max\{M_1f(x),M_2f(x)\}$ and so
		\begin{equation}\label{maxlocalnonlocal}
			O(Mf,Q_0)\lesssim O(M_1f,Q_0)+O(M_2f,Q_0).
		\end{equation}
		To estimate the first term in the right hand side of \eqref{maxlocalnonlocal}, let $Q_0^*$ be the concentric dilation of $Q_0$ with $l(Q_0^*)=2cl(Q_0)$. Then for the local part we have
		
		\begin{align*}
			&O(M_1f,Q_0)\le 2\ave{\left|M_1f-f_{Q_0^*}\right|}{Q_0}\le 2\ave{M_1\left|f-f_{Q_0^*}\right|}{Q_0}\\&\le 2\left(\ave{\left(M_1\left|f-f_{Q_0^*}\right|\right)^2}{Q_0}\right)^{\frac{1}{2}}\le 2\left(\frac{1}{|Q_0|}\int \left(M\left|f-f_{Q_0^*}\right|\chi_{Q_0^*}\right)^2\right)^{\frac{1}{2}}.
		\end{align*}
		By using the boundedness of $M$ on $L^2(\rn)$ we get
		\begin{align*}
			O(M_1f,Q_0)\lesssim c^{\frac{n}{2}}\left(\ave{\left|f-f_{Q_0^*}\right|^2}{Q_0^*}\right)^\frac{1}{2},
		\end{align*}
		and an application of the John-Nirenberg inequality gives us
		\begin{equation}\label{localestimate}
			O(M_1f,Q_0)\lesssim c^{\frac{n}{2}}\|f\|_{BMO(Q_0^*)}.
		\end{equation}
		
		To estimate the mean oscillation of the nonlocal part, suppose $x,y\in Q_0$, $M_2f(x)>M_2f(y)$ and let $Q$ be a cube with $l(Q)\ge cl(Q_0)$, which contains $x$ and such that $M_2f(y)<f_Q$. Now, let $Q'$ be a cube such that $Q_0\cup Q\subset Q'$, $l(Q')=l(Q)+l(Q_0)$, and let $A=Q'\backslash Q$. Then $M_2f(y)\ge f_{Q'}$ and we have
		
		\begin{align*}
			&f_Q-M_2f(y)\le f_Q-f_{Q'}=f_Q-\left(\frac{|A|}{|Q'|}f_A+\frac{|Q|}{|Q'|}f_Q\right)=\frac{|A|}{|Q'|}\left(f_Q-f_A\right)\\&\le\frac{|A|}{|Q'|}\left(\left|f_Q-f_{Q'}\right|+\left|f_{Q'}-f_A\right|\right)\lesssim \frac{|A|}{|Q|}O(f,Q')+\frac{|A|}{|Q'|}\ave{\left|f-f_{Q'}\right|}{A}.
		\end{align*}
		Here we note that $|A|=|Q'|-|Q|\approx l(Q_0)l(Q)^{n-1}$, and $l(Q')\approx l(Q)$. So from the above inequality and Lemma \ref{meanoscillatinonAlemma} we get
		
		\begin{equation*}
			f_Q-M_2f(y)\lesssim \frac{l(Q_0)}{l(Q)}\left(1+\log\frac{l(Q)}{l(Q_0)}\right)\lesssim c^{-1}\log c.
		\end{equation*}
		The reason for the last inequality is that $\frac{l(Q_0)}{l(Q)}\le c^{-1}$ and the function $-t\log t$ is increasing when $t<e^{-1}$. Finally, by taking the supremum over all such cubes $Q$, we obtain
		\begin{equation*}
			|M_2f(x)-M_2f(y)|\lesssim c^{-1}\log c, \quad x,y \in Q_0.
		\end{equation*}
		So for the nonlocal part we have
		\begin{equation}\label{nonlocalestimate}
			O(M_2f,Q_0)\le \dashint\limits_{Q_0}\dashint\limits_{Q_0}\left|M_2f(x)-M_2f(y)\right|dxdy\lesssim c^{-1}\log c.
		\end{equation}
		By putting \eqref{maxlocalnonlocal}, \eqref{localestimate} and \eqref{nonlocalestimate} together we get
		\begin{equation*}
			O(Mf,Q_0)\lesssim c^{\frac{n}{2}}\|f\|_{BMO(Q_0^*)}+c^{-1}\log c,
		\end{equation*}
		and taking the supremum over all cubes $Q_0$ with $l(Q_0)\le\delta$ gives us
		\begin{equation*}
			\omega(Mf,\delta)\lesssim c^{\frac{n}{2}}\omega(f,2c\delta)+c^{-1}\log c.
		\end{equation*}
		
		To finish the proof, it is enough to take the limit superior as $\delta\to0$ first and then let $c\to \infty$. 
	\end{proof}
	
	\begin{remark}
		The above argument shows that for all functions in $\bmorn$, if one chooses a sufficiently large localization of $M$, \eqref{nonlocalestimate} holds, meaning that the mean oscillation of the nonlocal part is small. This also shows itself in the dyadic setting: if one considers the dyadic maximal operator $M^d$ and dyadic $BMO$, denoted by $\bmodyadicrn$, then for a dyadic cube $Q_0$
		\[
		M^{d}_2f(x)=\underset{l(Q)\ge l(Q_0)}{\underset{x\in Q}\sup f_Q}=\sup_{Q_0\subset Q}f_Q, \quad x\in Q_0.
		\]
		Hence, $O(M^{d}_2f,Q_0)=0$ and therefore no dilation is needed $(c=1)$. 
	\end{remark}
	
	\section{Slices of BMO functions and unboundedness of directional and strong maximal operators}
	In this final section we discuss properties of slices of functions in $\bmorn$, and for simplicity we restrict ourselves to $\bmortwo$. We begin by asking:
	\begin{question*}
		Suppose $\varphi, \psi$ are two functions of one variable, when does  $f(x,y)=\varphi(x)\psi(y)$ belong to $\bmortwo$?
	\end{question*}
	
	To answer this, we need the following lemma which is an application of Fubini's theorem and its proof is found in \cite{MR4085448}.
	
	\begin{lemma}\label{meanosslicelemma}
		Let $A, B \subset \mathbb{R}$ be two measurable sets with finite positive measure, and $f$ be a measurable function on $\mathbb{R}^2$. Then
		\begin{equation*}
			O(f,A\times B)\approx \ave{O(f_y,A)}{B}dy+\ave{O(f_x,B)}{A}dx.
		\end{equation*}
	\end{lemma}
	
	Now, take two intervals $I, J$ with $l(I)=l(J)$. Then, an application of the above lemma to $f(x,y)=\varphi(x)\psi(y)$ gives us
	\begin{equation*}
		O(f,I\times J)\approx O(\psi,J)\ave{|\varphi|}{I}+O(\varphi,I)\ave{|\psi|}{J}.
	\end{equation*}
	Taking the supremum over all such $I,J$ we obtain
	\begin{equation}\label{phiandsicondition}
		\bmortwon{f}\approx \sup\limits_{\delta>0}\left(\sup_{l(I)=\delta}\ave{|\varphi|}{I}.\sup_{l(J)=\delta}O(\psi,J)+\sup_{l(I)=\delta}O(\varphi,I).\sup_{l(J)=\delta}\ave{|\psi|}{J}\right).
	\end{equation}
	When $f\in \bmortwo$ is non-zero on a set of positive measure, the above condition implies that $\varphi, \psi\in\bmor$. To see this, note that if $\varphi$ is non-zero on a set of positive measure, for some Lebesgue point of $\varphi$ like $x$, $\varphi(x)\ne0$. Then from the Lebesgue differentiation theorem, for sufficiently small $\delta$ we must have $\sup_{l(I)=\delta}\ave{|\varphi|}{I}\gtrsim |\varphi(x)|>0$. So $\psi$ has bounded mean oscillation on intervals with length less than $\delta$. For intervals $J$ with $l(J)\ge\delta$, $|\psi|$ has bounded averages because otherwise there is a sequence of intervals $J_n$ with $l(J_n)\ge\delta$ and $\lim\ave{|\psi|}{J_n}=\infty$. Then by dividing each of these intervals into sufficiently small pieces of length between $\frac{\delta}{2}$ and $\delta$, we conclude that $|\psi|$ has large averages over such intervals so $\sup_{l(J)=\delta}\ave{|\psi|}{J}=\infty$. But then $\sup_{l(I)=\delta}O(\varphi,I)=0$, which means that $\varphi$ is constant. We summarize the above discussion in the following proposition.
	\begin{proposition}\label{phiproposition}
		Let $f(x,y)=\varphi(x)\psi(y)$, $f\in\bmortwo$ if and only if \eqref{phiandsicondition} holds and if $f\ne 0$, then $\varphi, \psi \in\bmor$.
	\end{proposition}
	\begin{remark}
		When $\varphi$ and $\psi$ are not constants, the above argument shows that they belong to $bmo(\mathbb{R})$, the nonhomogeneous $BMO$ space, which is a proper subspace of  $\bmor$. See \cite{MR0523600} for more on $bmo(\mathbb{R})$.
	\end{remark}
	\begin{corollary}\label{loglog}
		Let $\log^-|x|=\max\{0, -\log |x|\}$ be the negative part of the logarithm and $p,q>0$ with $p+q\le1$. Then the function $f(x,y)=\left(\log^-|x|\right)^p\left(\log^-|y|\right)^q$ is in $\bmortwo$.
	\end{corollary}
	\begin{proof}
		A direct calculation shows that
		\begin{align*}
			&\sup_{l(I)=\delta}\ave{\left(\log^-|x|\right)^pdx}{I}\approx \begin{cases}
				\delta^{-1} & \delta\ge \frac{1}{2}\\
				\left(-\log\delta\right)^p & \delta< \frac{1}{2}
			\end{cases},\\
			&\sup_{l(J)=\delta}O\left(\log^-|\cdot|,J\right)^q\approx \begin{cases}
				\delta^{-1} & \delta\ge \frac{1}{2}\\
				\left(-\log\delta\right)^{q-1} & \delta< \frac{1}{2}
			\end{cases}.
		\end{align*}
		and the claim follows from Proposition \ref{phiproposition}.
	\end{proof}
	\begin{remark}
		The above function $f$ does not have bounded mean oscillation on rectangles, simply because the $BMO$-norm of the slices becomes larger and larger as we get closer to the origin. See \cite{MR2363526} (Example 2.32) for another example.
	\end{remark}
	Now we answer the third question of this paper.
	\begin{theorem}\label{strongmaximal}
		There exists a sequence of bounded functions $\{G_N\}$, $N\ge1$ such that it is bounded in $\bmortwo$ but
		\[
		\lim\limits_{N\to \infty}\bmortwon{M_{e_1}(G_N)}=\infty, \quad \lim\limits_{N\to \infty}\bmortwon{M_{s}(G_N)}=\infty.
		\]
	\end{theorem}
	To prove this we need the following simple lemma.
	\begin{lemma}\label{sumlemma}
		Let $Q_0=[-1,1]^n$, $f\in \bmorn$ with support in $Q_0$, and $x_k$ be a sequence in $\rn$ with $|x_k-x_m|\ge 3\sqrt{n}$ for $k\ne m$. Then $g(x)=\sum f(x-x_k)$ is in $\bmorn$ and $\bmornn{g}\lesssim \bmornn{f}$.
	\end{lemma}
	\begin{proof}
		First, by comparing the average of $|f|$ on $Q_0$ with $Q_0+2e_1$ we have
		\[
		\int_{Q_0}|f|=2^n\left(\ave{|f|}{Q_0}-\ave{|f|}{Q_0+2e_1}\right)\lesssim O\left(|f|,[-1,3]^n\right)\le \bmornn{|f|}\le2\bmornn{f}.
		\]
		Next, take a cube $Q$ and suppose for some $k$, $Q\cap (x_k+Q_0)\ne\emptyset$. We note that the distance of the support of functions $f(\cdot-x_k)$ from each other is at least $\sqrt{n}$ so if $l(Q)\le 1$ then $O(g,Q)=O(f(\cdot-x_k),Q)\le \bmornn{f}$. Otherwise, we have
		
		\begin{align*}
			O(g,Q)\le& 2 \ave{|g|}{Q}\le \frac{2}{|Q|}\sum_{Q\cap (x_k+Q_0)\ne\emptyset} \int_{x_k+Q_0}\left|f(y-x_k)\right|dy\\
			&\lesssim \frac{\#\{k| Q\cap (x_k+Q_0)\ne\emptyset\}}{|Q|} \bmornn{f}.
		\end{align*}

		Now to finish the proof, note that $\#\{k| Q\cap (x_k+Q_0)\ne\emptyset\}\lesssim |Q|$ which implies $O(g,Q)\lesssim \bmornn{f}$.
	\end{proof}
	
	\begin{proof}[Proof of Theorem \ref{strongmaximal}]
		We may assume $n=2$, since by a lifting argument similar to \eqref{lifting}, we can conclude the theorem for higher dimensions. Let $f$ be as in Corollary \ref{loglog}, $N$ a positive integer, and consider the following function:
		
		\begin{equation*}
			g_N(x,y)=\sum_{k=0}^{N}\sum_{m=2^k}^{2^{k+1}-1}f\left(x-3\sqrt{2}m,y-\frac{k}{N}\right).
		\end{equation*}
		$g_N$ has the following properties:\\
		
		(i) $\bmortwon{g_N}\lesssim 1$ (here our bounds only depend on $p,q$ but not $N$).\\
		
		This follows from Corollary \ref{loglog} and Lemma \ref{sumlemma} applied to $f$ with $x_{m,k}=\left(3\sqrt{2}m,\frac{k}{N}\right)$.\\

		(ii) $M_s(g_N)(x,y) \ge M_{e_1}(g_N)(x,y)\gtrsim\left(\log N\right)^q$ for $0\le x, y \le 1$ and $N\ge2$.\\
		
		To see this, let $0\le x\le1$ and $\frac{l}{N}\le y <\frac{l+1}{N}$ for some $l<N$. Then consider the average of $(g_N)_y$ on $I=[0,3\cdot2^{l+1}\sqrt{2}]$, which is bounded from below by
		\begin{equation*}
			\ave{(g_N)_y}{I}\ge \frac{1}{3\cdot2^{l+1}\sqrt{2}}\sum_{m=2^l}^{2^{l+1}-1}\int_{I} f\left(t-3\sqrt{2}m,y-\frac{l}{N}\right)dt\\.
		\end{equation*}
		Now note that for $2^l\le m \le 2^{l+1}-1$, $I$ contains the support of $f\left(\cdot-3\sqrt{2}m,y-\frac{l}{N}\right)$ and since $0\le y-\frac{l}{N}\le \frac{1}{N}$ we have
		\begin{align*}
			f\left(t-3\sqrt{2}m,y-\frac{l}{N}\right)\ge \left(\log N\right)^q\left(\log^-\left(t-3\sqrt{2}m\right)\right)^p.
		\end{align*}
		From this we get
		\begin{align*}
			M_{e_1}(g_N)(x,y)\ge\ave{(g_N)_y}{I} &\ge \frac{1}{3\cdot2^{l+1}\sqrt{2}}\sum_{m=2^l}^{2^{l+1}-1}\int_{I} f\left(t-3\sqrt{2}m,y-\frac{l}{N}\right)dt\\
			&\ge\frac{1}{6\sqrt{2}} \left(\log N\right)^q\int_{-1}^{1}\left(\log^-|t|\right)^pdt.
		\end{align*}
		
		At the end we note that for every function $g$,  $M_{e_1}(g)\le M_s(g)$	holds almost everywhere, and this proves the claim.\\
		
		(iii) $M_{e_1}(g_N)(x,y)=0$ for $y<-1$.\\
		
		This holds simply because $g_N$ is supported in $[3\sqrt{2}-1,\infty)\times[-1,2]$.\\
		
		(iv) $M_s(g_N)(x,y)\lesssim 1$ for $0\le x\le1, y\le -2$.\\
		
		To prove this final property of $g_N$, suppose $ R=I\times J$ is a rectangle with $(x,y)\in R$. Then if $R\cap supp(g_N)\ne \emptyset$ we have $l(I),l(J)\ge 1$, and we note that 
		\[
		\#\left\{(m,k)| R\cap supp\left(f\left(\cdot-3\sqrt{2}m,\cdot-\frac{k}{N}\right)\right) \ne\emptyset\right\}\lesssim l(I),
		\]
		which implies
		\[
		\ave{g_N}{R}\le l(I)^{-1}\#\left\{(m,k)| R\cap supp\left(f\left(\cdot-3\sqrt{2}m,\cdot-\frac{k}{N}\right)\right) \ne\emptyset\right\}\int_{\mathbb{R}^2} f \lesssim 1.
		\]
		Now taking the supremum over all rectangles $R$ proves the last property of $g_N$.\\
		
		Next, we measure the mean oscillation of $M_{e_1}(g_N)$ on the square $[-3,3]^2$ by
		\begin{equation*}
			O\left(M_{e_1}(g_N),[-3,3]^2\right)\gtrsim\ave{M_{e_1}(g_N)}{[0,1]^2}-\ave{M_{e_1}(g_N)}{[0,1]\times[-3,-2]}.
		\end{equation*}
		Then from the second and third properties of $g_N$ we obtain
		\begin{equation}\label{gdeltaoscillation}
			O\left(M_{e_1}(g_N),[-3,3]^2\right)\gtrsim \left(\log N\right)^q,
		\end{equation}
		and the same is true for $M_s$ by the third and fourth properties of $g_N$.\\
		
		At this point we note that the constructed sequence of functions $\{g_N\}$ has all the desired properties claimed in the theorem except that they are not bounded functions. However, this can be fixed by using a truncation argument as follows. For each $N,M\ge1$, let $g_{N,M}$ be the truncation of $g_N$ at height $M$, i.e.
		\[
		g_{N,M}:=\max\left\{M, \min\left\{g_N, -M\right\}\right\}.
		\]
		Next, we note that by the first property of $\{g_N\}$ this sequence is bounded in $\bmortwo$ and since $\bmortwon{g_{N,M}}\le4\bmortwon{g_N}$, the double sequence $\left\{g_{N,M}\right\}$ is also bounded in $\bmortwo$.
		Now, for each $N\ge1$, $g_N$ is a compactly supported function in $L^2(\mathbb{R}^2)$ and the sequence $\left\{g_{N,M}\right\}$ converges to $g_N$ in $L^2(\mathbb{R}^2)$ as $M$ goes to infinity. Then, since the operators $M_{e_1}$ and $M_s$ are continuous on this space, we conclude that for each $N\ge1$, $\left\{M_{e_1}(g_{N,M})\right\}$ and $\left\{M_s(g_{N,M})\right\}$ converge in $L^2(\mathbb{R}^2)$ to $M_{e_1}(g_N)$ and $M_s(g_N)$, respectively. Therefore, for $N'$ large enough (depending on $N$), we have
		\[
		O\left(M_{e_1}(g_{N,N'}),[-3,3]^2\right)\ge \frac{1}{2}O\left(M_{e_1}(g_N),[-3,3]^2\right)\gtrsim \left(\log N\right)^q,
		\]
		and
		\[
		O\left(M_{s}(g_{N,N'}),[-3,3]^2\right)\ge \frac{1}{2}O\left(M_{s}(g_N),[-3,3]^2\right)\gtrsim \left(\log N\right)^q.
		\]
		To finish the proof, let $G_N:=g_{N,N'}$ and note that $\left\{G_N\right\}$ is a sequence of bounded functions such that it is bounded in $\bmortwo$ but
		\[
		\lim\limits_{N\to \infty}\bmortwon{M_{e_1}(G_N)}=\infty, \quad \lim\limits_{N\to \infty}\bmortwon{M_{s}(G_N)}=\infty.
		\]
	\end{proof}
	By modifying the above function, one can construct a function in $\bmortwo$ such that none of its horizontal slices are in $\bmor$, which provides a negative answer to the forth question of this paper.\\
	
	\textbf{Example.} Let $\{r_m\}$ be an enumeration of the rational numbers and consider the following function:
	\begin{equation*}
		h(x,y)=\sum_{m\ge 1}f\left(x-3\sqrt{2}m,y-r_m\right).
	\end{equation*}
	Then we have
	\[
	O\left(h_y,[3\sqrt{2}m-1,3\sqrt{2}m+1]\right)=O\left(\left(\log^-\left(\cdot\right)\right)^p,[-1,1]\right)\left(\log^-\left(y-r_m\right)\right)^q.
	\]
	
	So by density of the rational numbers, for all values of $y$ we get $\sup_{l(I)=2} O(h_y,I)=\infty$, even though $h\in \bmortwo$.\\

	The above example shows that one can not control the maximum mean oscillation of the slices, when we look at intervals with a fixed length. However, in the following theorem, we show that there is a loose control when the length of intervals increases.
	
	\begin{theorem}
		Let $f\in\bmortwo$ with $\bmortwon{f}=1$. Then there exist constants $\lambda,c>0$, independent of $f$, such that for any sequence of intervals $I_k$ $(k\ge 1$)  with $l(I_k)=2^k$,  and any interval $J$ with $l(J)=1$, we have
		\begin{equation*}\label{exptheorem}
			\int_{J}e^{\lambda\sup_{k\ge 1}\frac{O\left(f_y,I_k\right)}{k}}dy\le c.
		\end{equation*}
	\end{theorem}
	
	\begin{proof}
		Let $E_t=\left\{y\in J|\sup\limits_{k\ge 1}\frac{O\left(f_y,I_k\right)}{k}>t\right\}$; then,
		\begin{equation}\label{etk}
			E_t=\bigcup_{k\ge1} E_{t,k},\quad \quad E_{t,k}=\left\{y\in J|\frac{O\left(f_y,I_k\right)}{k}>t\right\}.
		\end{equation}
		Now, taking the average over $E_{t,k}$ and applying Lemma \ref{meanosslicelemma} gives us
		\begin{equation*}
			t<\frac{1}{k}\ave{O\left(f_y,I_k\right)}{E_{t,k}}dy\lesssim \frac{1}{k}O\left(f,I_k\times E_{t,k}\right).
		\end{equation*}
		Next, let $J_k$ be the interval with the same center as J and with $l(J_k)=2^k$, and note that $E_{t,k}\subset J\subset J_k$, so $I_k\times E_{t,k}\subset I_k\times J_k$. Then an application of Lemma \ref{meanoscillatinonAlemma} shows that
		\[
		t\lesssim \frac{1}{k}O\left(f,I_k\times E_{t,k}\right)\lesssim\frac{1}{k}\left(1+\log\frac{|I_k\times J_k|}{|I_k\times E_{t,k}|}\right)\lesssim 1-\frac{1}{k}\log|E_{t,k}|.
		\]
		So for an appropriate constant $a>0$, which is independent of $f$, we have $|E_{t,k}|\lesssim e^{-atk}$ for $t>0$. From this and \eqref{etk} we get the estimate
		\begin{equation*}
			|E_t|\le\sum_{k\ge1} |E_{t,k}|\lesssim e^{-at}, \quad t>0.
		\end{equation*}
		Now, an application of Cavalieri's principle gives us
		\[
		\int_{J}e^{\frac{a}{2}\sup_{k\ge 1}\frac{O\left(f_y,I_k\right)}{k}}dy=\frac{a}{2}\int_{0}^{\infty}e^{\frac{a}{2} t}|E_t|dt\lesssim 1.
		\]
		Hence, \eqref{exptheorem} holds with $\lambda=\frac{a}{2}$, and this finishes the proof.
	\end{proof}
	
	\section*{Acknowledgements}
	I would like to thank G. Dafni and R. Gibara for suggesting these problems and valuable discussions and comments. Also, I am grateful to the referee for all the valuable comments on the manuscript.

	\bibliographystyle{plain}
	\bibliography{mybib}

\begin{thebibliography}{10}

\bibitem{MR0660603}
Colin Bennett.
\newblock Another characterization of {BLO}.
\newblock {\em Proc. Amer. Math. Soc.}, 85(4):552--556, 1982.

\bibitem{MR0621018}
Colin Bennett, Ronald~A. DeVore, and Robert Sharpley.
\newblock Weak-{$L\sp{\infty }$} and {BMO}.
\newblock {\em Ann. of Math. (2)}, 113(3):601--611, 1981.

\bibitem{MR0766959}
Sun-Yung~A. Chang and Robert Fefferman.
\newblock Some recent developments in {F}ourier analysis and {$H^p$}-theory on product domains.
\newblock {\em Bull. Amer. Math. Soc. (N.S.)}, 12(1):1--43, 1985.

\bibitem{MR4085448}
Galia Dafni and Ryan Gibara.
\newblock B{MO} on shapes and sharp constants.
\newblock In {\em Advances in harmonic analysis and partial differential equations}, volume 748 of {\em Contemp. Math.}, pages 1--33. Amer. Math. Soc., [Providence], RI, [2020] \copyright 2020.

\bibitem{MR0807149}
Jos\'{e} Garc\'{\i}a-Cuerva and Jos\'{e}~L. Rubio~de Francia.
\newblock {\em Weighted norm inequalities and related topics}, volume 116 of {\em North-Holland Mathematics Studies}.
\newblock North-Holland Publishing Co., Amsterdam, 1985.
\newblock Notas de Matem\'{a}tica, 104. [Mathematical Notes].

\bibitem{MR0523600}
David Goldberg.
\newblock A local version of real {H}ardy spaces.
\newblock {\em Duke Math. J.}, 46(1):27--42, 1979.

\bibitem{MR0131498}
F.~John and L.~Nirenberg.
\newblock On functions of bounded mean oscillation.
\newblock {\em Comm. Pure Appl. Math.}, 14:415--426, 1961.

\bibitem{MR0554817}
Peter~W. Jones.
\newblock Extension theorems for {BMO}.
\newblock {\em Indiana Univ. Math. J.}, 29(1):41--66, 1980.

\bibitem{MR2363526}
Anatolii Korenovskii.
\newblock {\em Mean oscillations and equimeasurable rearrangements of functions}, volume~4 of {\em Lecture Notes of the Unione Matematica Italiana}.
\newblock Springer, Berlin; UMI, Bologna, 2007.

\bibitem{MR4567945}
Giovanni Leoni.
\newblock {\em A first course in fractional {S}obolev spaces}, volume 229 of {\em Graduate Studies in Mathematics}.
\newblock American Mathematical Society, Providence, RI, [2023] \copyright 2023.

\bibitem{MR2280193}
Hannes Luiro.
\newblock Continuity of the maximal operator in {S}obolev spaces.
\newblock {\em Proc. Amer. Math. Soc.}, 135(1):243--251, 2007.

\bibitem{MR4012803}
Olli Saari.
\newblock Poincar\'{e} inequalities for the maximal function.
\newblock {\em Ann. Sc. Norm. Super. Pisa Cl. Sci. (5)}, 19(3):1065--1083, 2019.

\bibitem{MR0377518}
Donald Sarason.
\newblock Functions of vanishing mean oscillation.
\newblock {\em Trans. Amer. Math. Soc.}, 207:391--405, 1975.

\end{thebibliography}
	
\end{document}